\documentclass[12pt]{amsart}

\usepackage{amsmath}
\usepackage{amsfonts}
\usepackage{amssymb}
\usepackage{amsthm}
\usepackage{hyperref}
\usepackage{enumerate}
\usepackage{cleveref}
\usepackage{mathrsfs}
\usepackage[top=0.9in, bottom=1.15in, left=1.1in, right=1.1in]{geometry}
 \usepackage{hyperref}
\hypersetup{colorlinks=true,linkcolor=blue,citecolor=magenta}
\newtheorem{theorem}{Theorem}

\newtheorem{proposition}[theorem]{Proposition}
\newtheorem{corollary}[theorem]{Corollary}

\Crefname{conjecture}{Conjecture}{Conjectures}

\theoremstyle{plain}

\theoremstyle{plain}

\author{Maxwell Schneider and Robert Schneider}
\address{
Honors Program\newline
University of Georgia\newline
Athens, Georgia 30602}
\email{maxwell.schneider@uga.edu}
\address{Department of Mathematics\newline
University of Georgia\newline
Athens, Georgia 30602}
\email{robert.schneider@uga.edu}

\title{Digit sums and generating functions}
 
\begin{document}
\begin{abstract}
We connect a primitive operation from arithmetic --- summing the digits of a base-$B$ integer --- to $q$-series and product generating functions analogous to those in partition theory. We find digit sum generating functions to be intertwined with distinctive classes of ``$B$-ary'' Lambert series, which themselves enjoy nice transformations. We also consider digit sum Dirichlet series.
%
%
\end{abstract}

\maketitle


\section{Introduction and preliminaries}

For $B\geq 2$, let $s_B(a)$ denote the {\it digit sum} of a base-$B$ integer $a \geq 0$; for example, in base ten we have $s_{10}(73)=7+3=10$. Digit sums are interesting arithmetic objects, making numerous connections in number theory and combinatorics (see, for instance, \cite{AWR, Hurwitz, Kummer,qdigital, digitalbin2, Shallit, Tanay2, Tanay3, Watkins}). In this paper we apply $q$-series techniques to derive generating functions related to digit sums, and we prove other number-theoretic formulas. 

G. E. Andrews gives it as an exercise in \cite{Andrews} (p. 64, \#7) 
to prove for $B=10$ that 
\begin{equation}\label{eq0}
s_{B}(a)\equiv a\  (\operatorname{mod}\  B-1),
\end{equation}
which of course holds true for any base $B \geq 2$. 
Here we utilize exact identities 
analogous to 
congruence \eqref{eq0}, 
but involving a ``correction'' term $\widehat{c}_B$ to explicitly keep track of information about {\it carried numbers} (or ``carries'') from the traditional addition algorithm. 

%

%
%


We consider digit sums of the form $s_B(a_1+a_2+...+a_r)$ 
for positive integers $a_1,a_2,...,a_r$ written in base $B$. Let $c_B(a_1,a_2,...,a_r)$ denote the base-$B$ {\it carry sum}, the sum of all carried numbers produced in computing $a_1+a_2+...+a_r $ by the traditional addition algorithm. 

Then in the case $r=2$ (a sum of two integers), careful accounting of the steps of the addition algorithm produces the following identity, which the reader can readily verify
\footnote{Proposition \ref{theorem1} is equivalent to Theorem 2 in \cite{Kummer}; see also \cite{Watkins}.} 

\begin{proposition}\label{theorem1}
For base-$B$ integers $a,b\geq 0$, the digit sum of $a+b$ is given by 
\begin{equation*} 
s_B(a+b)=s_B(a)+s_B(b)- \widehat{c}_B(a,b)
\end{equation*} 
where $ \widehat{c}_B(a,b)=(B-1)c_B(a,b)$.
\end{proposition} 
We note the above proposition suggests a recursion relation. For $1\leq k\leq n$, write
\begin{flalign*} 
s_B(n)&=s_B(n-k)+s_B(k)- \widehat{c}_B(n-k,k)\\ 
&=s_B(n-2k)+2s_B(k)  -\widehat{c}_B(n-k,k)- \widehat{c}_B(n-2k,k) \\
&=s_B(n-3k)+3s_B(k)  -\widehat{c}_B(n-k,k)- \widehat{c}_B(n-2k,k)- \widehat{c}_B(n-3k,k) =... 
.\end{flalign*} 
Following this process through to the final, $\left \lfloor {n}/{k}\right\rfloor$th stage (where $\left \lfloor x \right \rfloor$ denotes the floor function of $x\in \mathbb R$) proves a family of identities for $s_B$ depending on the choice of $k$. 
\begin{corollary}\label{thmhat}
For $1\leq k \leq n$ we can write the base-$B$ digit sum as 
\begin{equation*}
s_B(n)=s_B\left(n- k\left \lfloor \frac{n}{k}\right\rfloor \right)+ s_B(k)  \left \lfloor \frac{n}{k}\right\rfloor-\sum_{i=1}^{\left \lfloor  n/k\right\rfloor } \widehat{c}_B(n-ik,k).
\end{equation*}

\end{corollary}
%
%

It also follows from the addition algorithm that we can extend Proposition \ref{theorem1} 
to the case of a sum of $r>2$ nonnegative base-$B$ integers $a_1, a_2, ..., a_r$. 
However, the generalization of the correction term $\widehat{c}_B$ is more complicated than just $(B-1)c_B$ when there are more than two summands, as we must keep up with carried numbers greater than $1$, which can yield arbitrarily large contributions to the digit sum. To this end, let us set 
\begin{equation}\label{correction0}
\widehat{c}_B\left( a_1,a_2,...,a_r \right):=\beta-s_B(\beta)+(B-1)c_B\left( a_1,a_2,...,a_r \right),\end{equation} 
in which $\beta$ denotes the {\it terminal 
carry} arising in the addition algorithm, the carried number left over from the sum of the left-most column of digits, which in turn comprises the left-most digits of the base-$B$ integer $a_1+a_2+...+a_r$.  


\begin{theorem}\label{proposition1} 
For base-$B$ integers $a_i \geq 0$, the digit sum of $a_1+a_2+...+a_r$ is given by  
\begin{equation*} 
s_B\left( \sum_{i=1}^{r} a_i \right)=\sum_{i=1}^{r }s_B(a_i)- \widehat{c}_B\left( a_1,a_2,...,a_r \right).
\end{equation*} 
\end{theorem}

\begin{proof}
The term $\widehat{c}_B$ is defined in \eqref{correction0} precisely to make the statement true: the roles of the $s_B(a_i)$ and carry sum $(B-1){c}_B$ arise naturally from the addition algorithm as in Proposition \ref{theorem1}, but an additional copy of the carry $\beta$ from the left-most column of the algorithm must be subtracted as it is ``brought down'' to form the left-most digits of the resulting sum; thus $s_B(\beta)$ contributes to the digit sum as well. 
\end{proof}

Note by comparison to \eqref{eq0} that since $s_B\left(a_1+a_2+...+a_r \right)\equiv a_1+a_2+...+a_r \equiv s_B(a_1)+s_B(a_2)+...+s_B(a_r)\  (\operatorname{mod}\  B-1)$, this generalized correction term still has the property 
\begin{flalign} \label{congr}
\widehat{c}_B\left( a_1,a_2,...,a_r \right)\equiv 0 \  (\operatorname{mod}\  B-1).
\end{flalign} 
Using Theorem \ref{proposition1} in the case that all $a_i$'s are equal, we can also find the digit sum of a {\it product} of two numbers. 
To state the result concisely, 
let us define 
\begin{equation}\label{correction} \widehat{c}_B(\left\{ a \right\}^b):= \widehat{c}_B(a,a,a,...,a)\end{equation} 
for the correction term 
when computing $a+a+a+...+a$ ($b$ times).
%

\begin{corollary}\label{proposition2}
For base-$B$ integers $a,b\geq 0$, the digit sum of the product $ab$ 
is given by 
\begin{flalign*} 
s_B(ab)=b\cdot s_B(a)- \widehat{c}_B(\left\{ a \right\}^b).
\end{flalign*} 
\end{corollary} 

Proposition \ref{theorem1} and Corollary \ref{proposition2} together show $s_B$ is linear up to correction terms (and thus by \eqref{congr} is a linear function modulo $B-1$). Comparing Corollary \ref{proposition2} with Corollary \ref{thmhat} ($n=ab, k=a$) gives a formula for 
$\widehat{c}_B(\{a\}^b)$ written in terms of simpler correction terms: 
%
\begin{equation}\label{corhat}
\widehat{c}_B(\{a\}^b)
=\sum_{j=0}^{b-1} \widehat{c}_B(aj,a).
\end{equation}
%
%
We note for $1\leq a \leq B-1$ the carry sum ${c}_B(a,a,...,a)$ ($b$ repetitions) has the value \begin{equation}\label{carry2}{c}_B(a,a,...,a)=\left \lfloor ab/B \right \rfloor.\end{equation}%

In fact, by the addition algorithm the above identities extend further, to sums of nonnegative {rational} numbers $a,b, a_i,$ etc., having terminating base-$B$ decimal expansions. 

\section{Digit sums and $q$-series}

It is natural to ask how digit sums connect to the theory of integer partitions 
 --- after all, the digits of integers represent weak integer compositions with parts strictly less than $B$ --- and thus to the study of $q$-series generating functions (see \cite{Andrews_theory}), where we take $q\in \mathbb C, |q|<1$, as usual. 
Andrews \cite{Andrewsmary}, Hirschhorn-Sellers \cite{HirschhornSellers} and others have studied closely-related $m$-ary partitions; and in recent work, Mansour-Nguyen \cite{qdigital} prove finite $q$-series involving $s_2(n)$ related to 
the $q$-binomial theorem\footnote{For example, summing both sides of Theorem 3 in \cite{qdigital} over $0\leq k\leq N$ and letting ${N\to \infty}$ gives a digit sum proof of the identity $\prod_{n=1}^{\infty}(1-q^n)\sum_{k=0}^{\infty}{q^{k(k-1)/2}}{\prod_{i=1}^{k}(1-q^i)^{-1}}=\prod_{n=0}^{\infty}(1+q^n)$.}, 
and Vignat-Wakhare \cite{Tanay3} give a number of interesting base-$B$ digit sum generating function relations. 
We seek product-sum generating functions of the $q$-hypergeometric variety (see \cite{Fine}). 
Classical $q$-series 
methods along the lines of \cite{Andrewsmary, HirschhornSellers} 
give a very nice 
two-variable identity. 

\begin{theorem}\label{twovarthm}
We have for $B\geq 2$ and $z\in \mathbb C, |q|<1$, that
\begin{flalign*} \sum_{n=0}^{\infty}q^n z^{s_B(n)}=\prod_{i=0}^{\infty}\frac{1-z^Bq^{B^{i+1}}}{1-zq^{B^i}}.
\end{flalign*}
\end{theorem}

\begin{proof}
It is evident by direct expansion of the right-hand product below for $f:\mathbb Z_{\geq 0} \to \mathbb C$ with $f(0):=1$, that we have a generating function for the unique partition of each $n\geq 1$ into distinct parts of the shape $mB^i$ for $1\leq m \leq B-1$, where no two parts have the same exponent $i\geq 0$ of $B$, i.e., a generating function for the base-$B$ expansion of $n$:
\begin{equation}\label{fundgenfctn}
\sum_{n=0}^{\infty}q^n\prod_{\text{digits $\nu_i$ of $n$}}f(\nu_i)=\prod_{i=0}^{\infty}\left(1+f(1)q^{B^i}+f(2)q^{2B^i}+...+f(B-1)q^{(B-1)B^i}     \right),
\end{equation}
in which the product on the left is taken over the base-$B$ digits $\nu_i$ of $n$ (i.e., $\nu_i\leq B-1$ is the nonnegative coefficient of $B^i$ for $0 \leq i\leq  \lfloor \operatorname{log}_B n \rfloor$). 
Taking $f(i)=z^i$, then recognizing the terms of the right-hand product as finite geometric series, gives Theorem \ref{twovarthm}. 
\end{proof}

Note as $s_B(n)=n$ when $n<B$, both sides of Theorem \ref{twovarthm} approach $(1-zq)^{-1}$ as  $B\to \infty$; similarly, setting $z=1$ reduces both sides to $(1-q)^{-1}$. Furthermore, one observes that taking $q\mapsto q^{B^j}, j\geq 0,$ effectively cancels the first $j$ terms of the product: 
\begin{equation}\label{funceq1}
\sum_{n=0}^{\infty}q^{nB^j} z^{s_B(n)}\  =\  \prod_{i=j}^{\infty}\frac{1-z^Bq^{B^{i+1}}}{1-zq^{B^i}}\  =\  \prod_{i=0}^{j-1}\frac{1-zq^{B^i}}{1-z^Bq^{B^{i+1}}}\sum_{n=0}^{\infty}q^{n} z^{s_B(n)}.
\end{equation}
In fact \eqref{funceq1} holds for any finite $j\in \mathbb Z$, not just $j\geq 0$. Other changes of variables result in further nice digit sum identities, like the following two-variable generating function.

\begin{corollary}\label{twovarthm2.5}
We have for $B\geq 2$ and $z\in \mathbb C, |q|<1$, that
\begin{flalign*} \sum_{n=0}^{\infty}q^n z^{\widehat{c}_B(\{ 1\}^n)}\  =\  \prod_{i=0}^{\infty}\frac{1-z^{B(B^i-1)}q^{B^{i+1}}}{1-z^{B^i-1}q^{B^i}}.
\end{flalign*}\end{corollary}

\begin{proof}
Taking $z\mapsto z^{-1},q\mapsto zq$ in Theorem \ref{twovarthm}, then Corollary \ref{proposition2} ($a=1, b=n$) gives  
\begin{flalign} \sum_{n=0}^{\infty}q^n z^{n-s_B(n)} =  \sum_{n=0}^{\infty}q^n z^{\widehat{c}_B(\{1\}^{n})}.
\end{flalign}
\end{proof}

Along similar lines, squaring both sides of Theorem \ref{twovarthm} 
gives another interesting identity.

\begin{corollary}\label{twovarcor}
We have for $B\geq 2$ and $z\in \mathbb C, |q|<1$, that
$$\sum_{n=0}^{\infty}q^nz^{s_B(n)}\sum_{k=0}^{n}z^{\widehat{c}_B(n-k,k)}=\prod_{i=0}^{\infty}\left( \frac{1-z^Bq^{B^{i+1}}}{1-zq^{B^i}}\right)^2.$$
\end{corollary}
 
\begin{proof}
After we square both sides of Theorem \ref{twovarthm}, the corollary results from applying the Cauchy product formula on the left-hand side in light of Proposition \ref{theorem1}: 
\begin{equation}
\left(\sum_{n=0}^{\infty}q^n z^{s_B(n)}\right)^2=\sum_{n=0}^{\infty}q^n \sum_{k=0}^{n} z^{s_B(n-k)+s_B(k)}=\sum_{n=0}^{\infty}q^n \sum_{k=0}^{n} z^{s_B(n)+\widehat{c}_B(n-k,k)}.
\end{equation}
\end{proof}

We observe in passing that the product on the right-hand side of Theorem \ref{twovarthm} can also be expressed in the following $q$-hypergeometric form\footnote{This formula looks like an extended version of identities used in \cite{HirschhornSellers}.}. 

\begin{theorem}\label{twovarthm3}
We have for $B\geq 2$ and $z\in \mathbb C, |q|<1$, that
\begin{flalign*}\prod_{i=0}^{\infty}\frac{1-z^Bq^{B^{i+1}}}{1-zq^{B^i}}\  =\  \frac{1}{1-zq} \  +\  \frac{z-z^B}{1-z^Bq}\sum_{n=1}^{\infty}q^{B^{n}}\frac{\prod_{j=0}^{n-1} (1-z^Bq^{B^{j}})}{\prod_{j=0}^{n}(1-zq^{B^j})}.
\end{flalign*}
\end{theorem}

\begin{proof}
Multiply the product on the left side by $(1-z^Bq)$ and expand the result as a telescoping series, which we write in this form:
\begin{flalign}\prod_{i=0}^{\infty}\frac{1-z^Bq^{B^i}}{1-zq^{B^i}}
= \frac{1-z^Bq}{1-zq} + \sum_{n=0}^{\infty}\frac{\prod_{j=0}^{n} (1-z^Bq^{B^j})}{\prod_{j=0}^{n}(1-zq^{B^j})}\left(\frac{1-z^Bq^{B^{n+1}}}{1-zq^{B^{n+1}}}-1\right).
\end{flalign}
Simplifying the summands on the right-hand side, then dividing both sides of the equation by $(1-z^Bq)$ and adjusting indices appropriately, completes the proof. \end{proof}

To give a $q$-series generating function for $s_B(n)$, we first 
introduce an auxiliary series $L_B(q)$ that represents a $B$-ary variant of classical Lambert series\footnote{A variant of $L_2(q)$ is studied in \cite{Tanay3}.} for $|q|<1$: 
\begin{equation}
L_B(q):=\sum_{i=1}^{\infty}\frac{q^{B^i}}{1-q^{B^i}}.
\end{equation}
Clearly $L_B(q)\to0$ as $B\to \infty$. We note in passing the interesting property of this series that, as in \eqref{funceq1}, taking $q\mapsto q^{B^j}$ effectively shifts the indices: 
\begin{equation}\label{shift}
L_B(q^{B^j})=\sum_{i=j+1}^{\infty}\frac{q^{B^i}}{1-q^{B^i}}\  \  \text{for any $j\in \mathbb Z$}.
\end{equation}
The series $L_B(q)$ 
is a primary component of 
the following generating function\footnote{We note that Corollary \ref{s_Bgen} is equivalent to Theorem 1 of \cite{AWR} and Theorem 12 of \cite{Tanay3}.} for $s_B(n)$. 

\begin{corollary}\label{s_Bgen}
We have for $B\geq 2$ and $|q|<1$, that 
\begin{equation*} 
\sum_{n=1}^{\infty}s_B(n)q^n\  =\  \frac{q}{(1-q)^2}-\frac{B-1}{1-q}L_B(q).
\end{equation*}
\end{corollary}
\

\begin{proof} To prove the theorem, we utilize the following two-variable extension\footnote{This two-variable function transforms just like \eqref{shift} under $q\mapsto q^{B^j}$.} of $L_B(q)$:
\begin{equation}\label{shift2}
\mathcal L_B(x;q)\  :=\  \sum_{i=0}^{\infty}\frac{q^{B^i}}{1-xq^{B^i}}\  =\  \sum_{n=0}^{\infty}x^n\sum_{i=0}^{\infty}q^{(n+1)B^i}
\end{equation}
where $|xq|<1$ and the right-hand sum results after recognizing the terms of the first sum as geometric series. 
Note the series defining $\mathcal L_B(x;q)$ begins at $i=0$ instead of $1$, thus \begin{equation}\label{shift3} L_B(q)=\mathcal L_B(1;q)-q/(1-q).\end{equation} Then applying the differential operator $z \frac{\text{d}}{\text{d}z}$ to both sides of Theorem \ref{twovarthm} yields 
\begin{equation}\label{identity} \sum_{n=1}^{\infty}s_B(n)q^n z^{s_B(n)}=\prod_{i=0}^{\infty}\frac{1-z^Bq^{B^{i+1}}}{1-zq^{B^i}}\left[z \mathcal L_B(z;q)-Bz^B \mathcal L_B(z^B; q^B) \right].\end{equation}
Letting $z=1$, all the numerators and denominators of the infinite product on the right cancel except the denominator  $(1-q)^{-1}$ of the $i=0$ factor. Noting by \eqref{shift} and \eqref{shift3} that \begin{equation} \mathcal L_B(1;q)-B\mathcal L_B(1;q^{B})=\frac{q}{1-q}-(B-1)L_B(q) \end{equation}
then completes the proof. 
\end{proof}

We point out both sides of the theorem approach $q \frac{\text{d} }{\text{d}q} (1-q)^{-1}$ as $B\to \infty$. Combining Corollary \ref{s_Bgen} with Proposition \ref{theorem1} shows that in fact $L_B(q)$ (multiplied by $B-1$) represents 
the generating function for the correction term $\widehat{c}_B(n-1, 1)$.\footnote{Without the multiplicative factor, $L_B(q)$ generates the carry sum $c_B(n-1,1)=\widehat{c}_B(n-1,1)/(B-1)$.}

\begin{corollary}\label{shiftcor} We have for $B\geq 2$ and $|q|<1$ that
$$\sum_{n=1}^{\infty} \widehat{c}_B(n-1,1)q^n =  (B-1)L_B(q).$$ 
\end{corollary}

\begin{proof} We simply note that setting $a=n-1, b=1$ in Proposition \ref{theorem1} gives
\begin{equation}\label{identity2}
(1-q)\sum_{n=1}^{\infty}{s}_B(n)q^n= \sum_{n=1}^{\infty} \left({s}_B(n)-{s}_B(n-1)\right)q^n=\sum_{n=1}^{\infty}q^n-\sum_{n=1}^{\infty}\widehat{c}_B(n-1,1)q^n.
\end{equation}
Comparing \eqref{identity2} with Corollary \ref{s_Bgen} gives the corollary. 
 \end{proof}
%

The series $L_B(q)$ also essentially generates the sequence of correction terms $\widehat{c}_B(\{1\}^n)$.

\begin{corollary}\label{c_Bgen}
We have for $B\geq2$ and $|q|<1$ that
\begin{equation*}
\sum_{n=1}^{\infty}\widehat{c}_B(\{1\}^n) q^n
=\frac{B-1}{1-q}L_B(q).\end{equation*}
\end{corollary}

\begin{proof}
It is equivalent to \eqref{corhat} that we have the recursion relation
\begin{equation}\label{recursion}
\widehat{c}_B(\{a\}^{b})=\widehat{c}_B(\{a\}^{b-1})+\widehat{c}_B(a(b-1), a).\end{equation}
Setting $a=1$, and noting $\widehat{c}_B(\{ a \}^0)=0$ holds for all $a\geq 0,B\geq 2$, we then can write
\begin{equation}\label{recursion2}
(1-q)\sum_{n=1}^{\infty}\widehat{c}_B(\{1\}^{n})q^n=\sum_{n=1}^{\infty}\left( \widehat{c}_B(\{1\}^{n})-\widehat{c}_B(\{1\}^{n-1})\right)q^n=\sum_{n=1}^{\infty}\widehat{c}_B(n-1, 1)q^n.\end{equation}
Comparing \eqref{recursion2} to Corollary \ref{shiftcor} completes the proof.
\end{proof}
As in the proof of Corollary \ref{s_Bgen}, the preceding corollary also equals the result of applying $z \frac{\text{d}}{\text{d}z}$ throughout Corollary \ref{twovarthm2.5}, then evaluating at $z=1$. Moreover, Corollary \ref{c_Bgen} highlights a special case of a more general identity tying together these two primary instances of $\widehat{c}_B$. 

\begin{theorem}
We have for $B\geq2$ and $|q|<1$ that
\begin{equation*}
\sum_{n=1}^{\infty}\widehat{c}_B(\{a\}^n) q^n
=\frac{q}{1-q}\sum_{n=0}^{\infty} \widehat{c}_B(an,a)q^n.
\end{equation*}
\end{theorem}

\begin{proof}
From \eqref{recursion} we can write more generally
\begin{equation}\label{recursion3}
(1-q)\sum_{n=1}^{\infty}\widehat{c}_B(\{a\}^{n})q^n=\sum_{n=1}^{\infty}\left( \widehat{c}_B(\{a\}^{n})-\widehat{c}_B(\{a\}^{n-1})\right)q^n=\sum_{n=1}^{\infty}\widehat{c}_B(a(n-1), a)q^n.\end{equation}
Adjusting the index of summation (viz. $n\mapsto n+1$) on the right leads to the theorem.
\end{proof}

One wonders if combinatorial interpretations for correction term identities such as these, as well as generating functions for the general case $\widehat{c}_B(a_1, a_2,..., a_r)$, may be found. 

In our study, we noticed further $q$-series connections we wish to follow up on. 
For instance, the type of index-shifting phenomenon highlighted in \eqref{shift} is naturally embedded in $q$-series related to $L_B(q)$ (and \eqref{fundgenfctn}); 
interrelations and congruences between coefficients may be found. 
We note that a similar shift of indices in a three-variable bilateral variant of $L_B(q)$:
\begin{equation}\widehat{\mathcal L}_B(x;z,q):=\sum_{n\in \mathbb Z}\frac{z^nq^{B^n}}{1-xq^{B^n}},\end{equation} valid if $|z|>B, x\neq q^{-B^n}$ for $n\in\mathbb Z$, yields for $r,t\in\mathbb Z, t\neq 0,$ functional equations such as 
\begin{flalign}\widehat{\mathcal L}_{B^{-1}}(x;z^{-1},q)=\widehat{\mathcal L}_B(x;z,q),  \  \  \  \  \  \  \  \  \  \    z^r\widehat{\mathcal L}_B(x;z,q^{B^r})=\widehat{\mathcal L}_B(x;z,q)\end{flalign}
and 
 \begin{flalign}
  z^r\widehat{\mathcal L}_{B^t}(x;z^t,q^{B^r})=\sum_{n \equiv r\  (\text{mod}\  t)}\frac{z^nq^{B^n}}{1-xq^{B^n}}.
 \end{flalign}
Fleeting explorations produced other nice relations, suggesting to the authors $\widehat{\mathcal L}_B(x;z,q)$ may be worthy of further study.

Connections to $q$-series such as those in \cite{qdigital}, \cite{Tanay3}, et al., and those given here, 
show that digit sums fit nicely into number theory and combinatorics in the neighborhood of partitions, unimodal sequences, compositions and other summatory structures. 

\section{Other digit sum connections}

In this section we briefly point out connections 
between digit sums and other generating functions from 
number theory (see \cite{AWR, Tanay3} for many further examples). 
If we define 
the divisor sum\footnote{This is an analog of 
the classical {\it sum of divisors function} $\sigma(n)$. Note $S_B(n)$ is {equal} to $\sigma(n)$ if $n<B$, and generally, the difference between the functions (by Corollary \ref{proposition2}) is $\sigma(n)-S_B(n)=\sum_{d|n}\widehat{c}_B(\{ 1\}^d)$.} 
%
%
\begin{equation}\label{partC}
S_B(n):=\sum_{d|n}s_B(d),
\end{equation}
it 
follows from standard methods\footnote{E.g. see p. 300, Prob. 11 of \cite{BB}} that for $|q|<1$, the {\it Lambert series generating function} for $s_B(n)$ is the $q$-series generating function for $S_B(n)$:
\begin{equation}
\sum_{n=1}^{\infty}\frac{s_B(n) q^n}{1-q^n}=\sum_{n=1}^{\infty}S_B(n)q^n.
\end{equation}
Similarly, classical Dirichlet convolution formally connects the {\it Dirichlet series generating functions} for $s_B(n)$ and $S_B(n)$ to the {\it Riemann zeta function} $\zeta(s):=\sum_{n=1}^{\infty}1/n^s$, viz. 
%
\begin{equation}\label{dirichlet2}
     \zeta(s) \sum_{n=1}^{\infty}\frac{s_B(n)}{n^s}= \sum_{n=1}^{\infty}\frac{S_B(n)}{n^s},\end{equation}
although the convergence of these series is not obvious {\it a priori}. Now, $\zeta(s)$ converges 
for $\operatorname{Re}(s)>1$, but since $1\leq s_B(n)\leq s_{B'}(n) \leq n$ for $B<B'$, 
convergence of \eqref{dirichlet2} also depends on the base $B$ in each instance. Using the fact $s_B(n)=n$ for $n<B$ gives 
\begin{equation}
\lim_{B\to\infty}\sum_{n=1}^{\infty}\frac{s_B(n)}{n^s}=\zeta(s-1), 
\end{equation}
so we may uniformly assume $\text{Re}(s)>2$ to ensure convergence in these series 
at any $B\geq 2$.

Many interesting identities relating $s_B(n)$ 
to 
$\zeta(s)$ as well as the {\it Hurwitz zeta function}
\begin{equation}\label{hurwitz} \zeta(s,x):=\sum_{n=0}^{\infty}\frac{1}{(n+x)^s}\  \  \  \  (x\in (0,1], \operatorname{Re}(s)>1)\end{equation}
can be found in the literature (for instance, see  \cite{Hurwitz, Tanay3}). Below we prove 
a Dirichlet series identity\footnote{We note that in light of Proposition \ref{theorem1}, Theorem \ref{cthm1} is equivalent to Corollary 2 of \cite{Hurwitz}.} related to the correction term $\widehat{c}_B$, as well as one related to the carry sum $c_B$. 

%
\begin{theorem} \label{cthm1}
We have for $\operatorname{Re}(s)>2$ that
$$ \sum_{n=1}^{\infty}\frac{\widehat{c}_B(n-1,1)}{n^s}= \frac{B-1}{B^s-1}\zeta(s).$$
\end{theorem}

\begin{proof}
We begin by observing that, for $t=\lfloor \operatorname{log}_B n \rfloor$ 
the highest power of $B$ in the base-$B$ expansion of $n\geq 1$, we have 
\begin{equation}
\widehat{c}_B(n-1,1)
= (B-1)\sum_{i=1}^{t}\left \lfloor \frac{n}{B^i}\right\rfloor= \left\{
        \begin{array}{ll}
            (B-1)N & \text{if $n=mB^N$ with $B\nmid m$,}\\
            
            0 & \text{otherwise.}
        \end{array}
    \right.
\end{equation} 
Then we can rewrite
\begin{equation}
\sum_{n=1}^{\infty}\frac{\widehat{c}_B(n-1,1)}{n^s}=(B-1)\sum_{N=1}^{\infty}\sum_{\substack{k\geq 1 \\ B \nmid k}}\frac{N}{(kB^N)^s}=(B-1)\left(1-\frac{1}{B^s}\right)\zeta(s) \sum_{N=1}^{\infty}\frac{N}{B^{sN}}.
\end{equation}
Recognizing the sum on the far right side as $u \frac{\text{d}}{\text{d}u}(1-u)^{-1}=u/(1-u)^2$ evaluated at $u=1/B^s$, and simplifying appropriately, yields 
Theorem \ref{cthm1}. 
%
\end{proof}

We let $c_B(\{1\}^n)$ denote the carry sum $c_B(1,1,...,1)$ ($n$ repetitions) for the next result.

\begin{theorem}\label{dir2}
We have for $\operatorname{Re}(s)>2$ that
$$ \sum_{n=1}^{\infty}\frac{c_B(\{1\}^n)}{n^s}= \frac{1}{B}\zeta(s-1)-\frac{1}{B^{s+1}}\sum_{k=1}^{B-1} k   \zeta(s,k/B).$$

\end{theorem}

\begin{proof}
Let $\text{frac}(x):=x-\lfloor x \rfloor$ denote the fractional part of $x\in \mathbb R$. Then by \eqref{carry2} we have 
\begin{flalign}\label{fraceq}
 \sum_{n=1}^{\infty}\frac{c_B(\{1\}^n)}{n^s}= \sum_{n=1}^{\infty}\frac{\lfloor n/B \rfloor}{n^s}=\frac{1}{B}\sum_{n=1}^{\infty}\frac{1}{n^{s-1}}-\sum_{n=1}^{\infty}\frac{\text{frac}(n/B)}{n^s}.
\end{flalign}
If we set $n=mB+k$ for $m\geq 0,  0\leq k \leq B-1,$ then ${\text{frac}(n/B)}=k/B$ and we can write
\begin{equation}\label{fraceq2}
\sum_{n=1}^{\infty}\frac{\text{frac}(n/B)}{n^s}=\sum_{k=1}^{B-1}\sum_{m=0}^{\infty}\frac{k/B}{(mB+k)^s}=\frac{1}{B^{s+1}}\sum_{k=1}^{B-1}k\sum_{m=0}^{\infty}\frac{1}{(m+k/B)^s}
.\end{equation}
Comparing \eqref{hurwitz}, \eqref{fraceq} and \eqref{fraceq2} gives the theorem.
\end{proof}


\section{Computing the correction term}
A slightly opaque aspect of these formulas is the presence of $\widehat{c}_B$, for the $s_B$ are computed from digits prescribed at the outset and thus feel like reasonably concrete objects, 
whereas 
carried numbers only arise subsequently during the addition procedure so they feel less tangible. Of course, we can solve identities from Section 1 for the $\widehat{c}_B$ terms. 
%
%
Alternatively, here we note a recursive algorithm to explicitly compute the carry sum $c_B$, as well as the terminal carry $\beta$, and thus to write $\widehat{c}_B(a_1, a_2, ..., a_r)$ directly from the digits of $a_1, a_2,...,a_r$. 

Let $\delta_i\geq 0$ denote the number carried over from the $B^i$ column ($i\geq 0$) in the addition algorithm when computing $a_1+a_2+...+a_r$. Let $t$ denote the maximum power of $B$ occurring in the base-$B$ expansions of all the $a_i$, viz. the maximum value of $t_i$ over all 
\begin{equation}\label{a_i}
a_{i}=\alpha_{i,t_i} B^{t_i}+\alpha_{i,t_i-1} B^{t_i-1}+...+\alpha_{i,2} B^{2}+\alpha_{i,1} B+\alpha_{i,0} \end{equation} with integer {digits} $0\leq \alpha_{i,j} <B$. 
Then one can 
write the carries $\delta_0, \delta_1, \delta_2,..., \delta_t$ recursively:
\begin{flalign}
\delta_0=\left \lfloor \frac{\alpha_{1,0}+\alpha_{2,0}+...+\alpha_{r,0}}{B} \right\rfloor
\end{flalign}
and, for $1\leq j\leq t$,
\begin{flalign}
\delta_j&=\left \lfloor \frac{\alpha_{1,j}+\alpha_{2,j}+...+\alpha_{r,j}+\delta_{j-1}}{B} \right\rfloor,
\end{flalign}
where the $\alpha_{i,j}$ are the base-$B$ digits of the integer $a_i$ as in \eqref{a_i}. 
The reader can verify that the right-hand sides above do equal the carried numbers $\delta_j$. Then we have the following. 

\begin{proposition} 
For integers $a_i \geq 0$, the base-$B$ carry sum resulting from the computation of $a_1+a_2+...+a_r$ is given by
$$c_B(a_1, a_2, ..., a_r)=\left \lfloor \frac{\alpha_{1,0}+\alpha_{2,0}+...+\alpha_{r,0}}{B} \right\rfloor+\sum_{j\geq1}\left \lfloor \frac{\alpha_{1,j}+\alpha_{2,j}+...+\alpha_{r,j}+\delta_{j-1}}{B} \right\rfloor.$$
Moreover, the terminal carry $\beta$ is $\delta_t$.
\end{proposition}

Using this proposition and the preceding algorithm, together with the definition \eqref{correction0} of $\widehat{c}_B$, the computation of $\widehat{c}_B(a_1,a_2,...,a_r)$ from the digits of $a_1, a_2,...,a_r$ is straightforward.
%



%

\section*{Acknowledgments}
We wish to thank the anonymous referee for useful suggestions that strengthened this work. The first author was partially funded by HCSSiM at Hampshire College during the preparation of this paper; special thanks to Prof. David Kelly.


\begin{thebibliography}{BrStr}
\bibitem{AWR} F. T. Adams-Watters, and F. Ruskey. ``Generating functions for the digital sum and other digit counting sequences.'' {\it Journal of Integer Sequences} {\bf 12.2} (2009): 3.
\bibitem{Hurwitz} Jean-Paul Allouche and J. Shallit. ``Sums of digits and the Hurwitz zeta function.'' {\it Analytic Number Theory}. Springer, Berlin, Heidelberg, 1990. 19-30.
\bibitem{Andrewsmary} G. E. Andrews. ``Congruence properties of the $m$-ary partition function.'' {\it J. Number Theory} {\bf 3.1} (1971): 104-110.
\bibitem{Andrews_theory} G. E. Andrews. {\it The Theory of Partitions}, Encyclopedia of 
Mathematics and its Applications, vol. 2,  Addison--Wesley, Reading, MA, 1976.
Reissued, Cambridge University Press, 1998.
\bibitem{Andrews} G. E. Andrews. {\it Number theory}. Dover, 1994.
\bibitem{Kummer} T. Ball, T. Edgar and D. Juda. ``Dominance orders, generalized binomial coefficients, and Kummer's theorem.'' {\it Mathematics Magazine} {\bf 87.2} (2014): 135-143.
\bibitem{BB} J. M. Borwein, and P. B. Borwein. {\it Pi and the AGM}. Wiley, New York, 1987.
\bibitem{Fine} N. J. Fine. {\it Basic hypergeometric series and applications}. American Math. Soc., 1988.
\bibitem{HirschhornSellers} M. D. Hirschhorn and J. A. Sellers. ``A different view of $m$-ary partitions.'' {\it Australasian Journal of Combinatorics} {\bf 30} (2004): 193-196.
\bibitem{qdigital} T. Mansour and H. D. Nguyen. ``A $ q $-digital binomial theorem.'' arXiv preprint {\it arXiv:1506.07945} (2015).
\bibitem{digitalbin2}  H. D. Nguyen. ``A generalization of the digital binomial theorem.'' {\it J. Integer Sequences} {\bf 18.2} (2015): 3.
\bibitem{Shallit} J. O. Shallit. ``On infinite products associated with sums of digits." {\it J. Number Theory} {\bf 21.2} (1985): 128-134.
\bibitem{Tanay2} C. Vignat and T. Wakhare. ``Base-$b$ analogues of classic combinatorial objects.'' arXiv preprint {\it arXiv:1607.02564} (2016).
\bibitem{Tanay3} C. Vignat and T. Wakhare. ``Finite generating functions for the sum-of-digits sequence.'' {\it Ramanujan Journal} (2018): 1-46.
\bibitem{Watkins} T. Watkins. ``Digit sum arithmetic.'' Applet-magic.com {\it http://www.sjsu.edu/faculty/watkins/} {\it Digitsum.htm}.

\end{thebibliography}
\end{document}